\documentclass[12pt]{amsart}
\usepackage{amsmath}
\usepackage{amsfonts}
\usepackage{amssymb}
\usepackage{a4wide}
\usepackage{comment}
\usepackage{bookmark}
\usepackage{graphicx}
\usepackage{xcolor}

\def \F {{\mathbb F}}

\def \cV {{\mathcal V}}
\def \cB {{\mathcal B}}
\def \aind {\mathrm{AddInd}}
\def \codim {\mathrm{codim}\,}
\def \im {\mathrm{im}\,}

\def \crk {\mathrm{Crk}}
\def \Crk {\mathrm{Crk}}
\def \w {\mathrm{w}}

\usepackage{todonotes}

\newtheorem{theorem}{Theorem}

\newtheorem{lemma}{Lemma}

\newtheorem{proposition}{Proposition}
\newtheorem{conj}{Conjecture}
\theoremstyle{remark}

\newtheorem*{remark}{Remark}
\newtheorem*{remarks}{Remarks}
\newcommand{\thmtext}{Suppose a permutation $f$ of $\F_q$ has Carlitz rank $r>0$.
    Then its additive index is at least $\min \left\{\frac{q}{r},\frac{q-2r}{2}\right\}$.
    Further if $r=1$ and $q> 4$, then $\aind(f)=q$.}

\date{\today}

\begin{document}
\title{Additive index and Carlitz rank}

\author{Pierre-Yves Bienvenu}

\author{Arne Winterhof}

\email{\{pierre.bienvenu,arne.winterhof\}@oeaw.ac.at}

\address{Johann Radon Institute for Computational and Applied Mathematics, Austrian Academy of Sciences, Linz, Austria}

\begin{abstract}
We compare several complexity measures for self-mappings of finite fields. In particular,
we show that Carlitz rank and additive index cannot be small simultaneously up to trivial exceptions. That is, these two measures detect cryptographic weaknesses of different classes of functions.

We also study the relationship between additive index and degree or weight, respectively,
complementing earlier results of Aksoy et al.\ and Gómez-P\'erez et al.\ on the relationship between Carlitz rank and degree or weight, respectively. 

Finally, we show that a function closely related to the discrete logarithm provides an example in which all four complexity measures, degree, weight, additive index and Carlitz rank, are large. 
\end{abstract}

\maketitle

\section{Introduction}
For the whole paper, we fix a prime number $p$ and an integer $n\ge 1$ and put $q=p^n$. We denote by $\F_q$ the finite field of $q$ elements.

Several complexity measures such as degree, weight and multiplicative index for self-mappings of $\F_q$ are relevant in cryptography.
An interesting question is to determine to which extent they can be simultaneously small,
in other words, whether a function can be  simple from two different viewpoints at once.
Given a function $f:\F_q\rightarrow\F_q$, we will consider the following complexity measures: the {\em degree} $\deg (f)$ -- which
is the degree of the unique polynomial 
$$P_f(x)=\sum_{i=0}^{q-1}a_ix^i\in\F_q[x]$$ of degree
at most $q-1$ coinciding with $f$ on $\F_q$; the {\em weight} (or {\em sparsity}) $\w(f)$, which is the number of nonzero coefficients of $P_f$; and finally the {\em Carlitz rank} and the {\em additive index}, which we will now define.

\subsection{The Carlitz rank}
The Carlitz rank was introduced by Aksoy et al.\ in \cite{ak09} and is defined as follows.
By a well-known result of Carlitz \cite{ca53}, the group of permutations (under composition) of $\F_q, q\geq 3$, is generated by the set consisting of the inversion 
$$x\mapsto x^{q-2}=\left\lbrace\begin{array}{cc}
0    ,    &x=0  \\
    x^{-1} & x\neq 0
\end{array} \right.$$ 
and all affine-linear maps $x\mapsto\alpha x+\beta$, $\alpha,\beta\in\F_q,\alpha\neq 0$.
Accordingly, the {\em Carlitz rank} $\crk(f)$ of a permutation $f$ of $\F_q$ is the minimum number of inversions required to write $f$ as a composition of the inversion and affine maps. In other words, any permutation $f$ of~$\F_q$ may be represented by a polynomial of the form
$$    P_n(a_0,a_1,\ldots,a_{n+1};x)=\left(\cdots ((a_0x+a_1)^{q-2}+a_2)^{q-2}+\cdots+a_n\right)^{q-2} +a_{n+1}
$$
where $a_i\neq 0$ for $i=0,2,\ldots,n$ (note that $a_1a_{n+1}$ may be 0)
and $\Crk(f)$ is the smallest integer $n\geq 0$ making such a representation possible.
In particular $\Crk(f)=0$ if, and only if, $f$ is an affine-linear map, that is, of the
form $f(x)= ax+b$.

For a survey on Carlitz rank see \cite{to14}.

\subsection{The additive index}
This notion was introduced by Reis and Wang in \cite{rewa22}.
A polynomial over $\F_{q}$ of the form
$$M(X)=\sum_{j=0}^{n-1}a_jX^{p^j},\quad a_j \in \F_q,$$
is called {\em linearised} (or {\em $p$-polynomial}).
Note that $M(X)$ is an $\F_p$-linear map if we consider~$\F_{q}$ a vector space, that is,
$$M(a\xi+b\zeta)=aM(\xi)+bM(\zeta),\quad a,b\in \F_p,\quad \xi,\zeta \in \F_{q}.$$
In fact a simple counting argument shows that any $\F_p$-endomorphism of $\F_q$ ($\F_p$-linear self-mapping of $\F_q$) is induced by a linearised polynomial.
See \cite{biwi26,rewa22} and the references therein for more details, background and motivation.

A mapping $f:\F_q\rightarrow\F_q$ is of {\em additive index at most $p^k$}
and {\em codimension} at most $k$ 
if there is a linearised polynomial $M(X)\in \F_{q}[X]$ 
and a linear subspace $U_0$ of dimension $n-k$ with disjoint cosets~$U_i$, $i=0,1,\ldots,p^k-1$, such that
$$f(\xi)=M(\xi)+a_i,\quad \xi\in U_i,\quad i=0,1,\ldots,p^k-1,$$
for some $a_0,a_1,\ldots,a_{p^k-1}\in \F_{q}$, see  \cite[Definition~4.2]{rewa22}.
The codimension of $f$, denoted $\codim f$, is then the smallest $k$ such that $f$ is of codimension at most $k$,
and the additive index is $\aind(f)=p^{\codim f}$.

\begin{remark}
    The counting argument alluded to above reveals that $M$ may be assumed to have degree 
    at most
    $p^{n-k-1}$, see \cite[Corollary 7]{biwi26}.
\end{remark}

\subsection{Previous results on the relation between these complexity measures}
Let~$f$ be a permutation of $\F_q$ and assume $\deg(f)\ge 2$.
Aksoy et al.\ \cite[Theorem 4]{ak09} showed that for any permutation $f$ of $\F_q$
\begin{equation}\label{deg}
  \crk(f)\ge q-\deg(f)-1,\quad \deg(f)>1.
\end{equation}
Thus, the Carlitz rank and the degree cannot be small simultaneously, except for affine maps.

Further, under the assumption that $\deg(f)>1$ and $f(x)$ is not of the form 
$$f(x)=a+bx^{q-2}, \quad b\not=0,$$ 
Gómez-P\'erez et al. \cite[Theorem 4]{goosto14} proved that
\begin{equation}\label{weight}
\Crk(f)>\frac{q}{\w(f)+2},\quad \Crk(f)>0,\quad f(x)\not=a+bx^{q-2}.
\end{equation}
Again, Carlitz rank and weight cannot be small simultaneously, except for $\Crk(f)\in \{0,1\}$.

\subsection{New results in this paper}
First, we prove that a permutation of degree at least $2$ and small positive Carlitz rank necessarily has high additive index. In particular, we have
\begin{equation}
\label{rankadd}
\aind(f)> \frac{q}{\crk(f)+1}-1,\quad \deg(f)>1,
\end{equation}
which follows from the more precise Theorem~\ref{th:lowrank} in Section~\ref{main} below.


Additionally, we prove results relating the additive index to the degree and the weight.
Indeed, we prove that any self-mapping $f$ of $\F_q$ of positive codimension
satisfies 
$$\deg(f)\aind(f)\geq q,$$ 
and this bound is sharp, see Section~\ref{degree}.
The relation weight--additive index is more complex. In contrast to other measures,
we find that the weight is also small when the additive index is small; but when the additive index is large, the weight can be more or less anything, see Section~\ref{secweight}.

In Section~\ref{log} we show that a permutation of $\F_q$ which can be identified with the discrete logarithm, and which was already known to have large
degree, large weight and large additive index, also has large Carlitz rank, providing an explicit, cryptographically relevant example of a permutation having all four complexity measures studied here large simultaneously.

In Section~\ref{multind} we recall the known analogues of our results for the multiplicative index and state a conjecture about the relation between additive and multiplicative index.

\section{Preliminaries}
We provide some results that will be needed later.
\subsection{Weight and zeros}

We start with \cite[Lemma 2.5]{sh03}.
\begin{lemma}\label{wt}
Let $f(x)\in \F_q[x]$ be a non-zero polynomial of degree at most $q-2$
and of weight $t\ge 1$. Then there are at least $\frac{q - 1}{t}$ values of $c\in \F_q^*$
with $f(c)\ne 0$.
\end{lemma}

\subsection{Interpolation by linear rational functions}
 The following result from \cite[Section 2]{to14} is useful
 for understanding permutations of low Carlitz rank. 
 \begin{lemma}
 \label{cr}
     Let $f$ be a permutation of $\F_q$ of Carlitz rank $r$. Then
 there are $\alpha,\beta,\gamma,\delta \in \F_q$ such that
    $$f(x)=\frac{\alpha x+\beta}{\gamma x+\delta}$$
    for at least $q-r$ different $x\in \F_q$.
 \end{lemma}

\subsection{Equivalent definition of the codimension of a mapping}

For the following, see also \cite{rewa22}.
\begin{proposition}
\label{th:coset}
Let $U$ be a subspace of $\F_q$.
    The polynomial 
    $$h(x)=\prod_{u\in U}(x-u)\in\F_q[x]$$ 
    is a linearised polynomial.
    Further, any function $f:\F_q\rightarrow\F_q$ which is constant on each coset of $U$
    may be written $f=g\circ h$, where $\deg g<\frac{q}{|U|}$.
\end{proposition}
\begin{proof}
    Naturally $h$ is constant on each coset of $U$ and is of degree $|U|$. Therefore
    $h-h(v)$ for any $v\in\F_q$, does not vanish outside $U+v$; thus two distinct cosets
    of $U$ are mapped to two distinct values by $h$. This implies that any function which
     is constant on each coset of $U$
    may be written $f=g(h(x))$. Also note that $\{h(x):x\in\F_q\}$ has cardinality $q/|U|$ (the number of cosets of $U$), so $g$ may be chosen to have degree $<q/|U|$ by Lagrange interpolation.

    We now show that $h$ is linearised.
    Fix $y\in\F_q$. We want to show that the polynomial
    $$\ell(x)=h(x+y)-h(x)-h(y)$$ 
    is the zero polynomial.
    To do this, it suffices to observe that $\deg \ell<\deg h=|U|$ and $\ell$ vanishes on $U$.
\end{proof}
\begin{proposition}\label{alt}
        Given a self-mapping $f$ of $\F_q$ and an integer $k\geq 0$, the following are equivalent.
    \begin{enumerate}
        \item $f$ has codimension at most $k$.
        \item There exist linearised polynomials $M,L\in\F_q[x]$ 
        such that 
        \begin{itemize}
            \item 
        $M$ splits completely over $\F_q$ and $\deg M\geq p^{n-k}$,
        \item  $\deg L<\deg M$,
        \item $f(x)=g(M(x))+L(x)$ where $g$ is of degree at most $p^{k}-1$. \end{itemize}
    \end{enumerate}
\end{proposition}
\begin{proof}
    If $f$ has codimension at most $k$, there exists a subspace $U\leq \F_q$ of codimension at most $k$ and a linearised polynomial $L$ of degree at most $p^{n-k-1}$ such that
    $f(x)-L(x)$ is constant on every coset of $U$. Therefore
    by Proposition \ref{th:coset}, we have $f(x)=L(x)+g(M(x))$, where $\deg g\leq p^k-1$ and $M$ is linearised of degree $p^{n-k}$.

    Conversely, suppose that there exist linearised polynomials $M,L\in\F_q[x]$ with the properties as above. Then consider the subspace $U=\ker M$; then $|U|=\deg M\geq p^{n-k}$, hence $U$ has codimension at most $k$. Then $g\circ M$ is constant on each coset of $U$. Thus $f(x)=L(x)+a_{U+x}$ for some mapping $a:\F_q/U\rightarrow \F_q$.
\end{proof}
\subsection{Behaviour of the additive index under composition}
Finally, we state a useful simple lemma about the behaviour of the additive index with respect to composition. 
\begin{lemma}
\label{lm:compoGen}
    Let $f,g,h$ be self-mappings of $\F_q$ such that $f=g\circ h$. Then the additive index of $f$ is no larger than the product of the indices of $g$ and $h$, or equivalently,
    $$\codim f\leq \codim g+\codim h.$$
\end{lemma}
\begin{proof}
    By definition of the additive index, there exist $\F_p$-linear maps $L,M$, subspaces $U,V$ of $\F_q$ of codimension $\codim g,\codim h$ respectively and
    maps $a:\F_q/U\rightarrow\F_q,b:\F_q/V\rightarrow\F_q$
    such that $g(x)=L(x)+a(U+x)$ and $h(x)=M(x)+b(V+x)$.

    Therefore \begin{eqnarray*}
        f(x)&=&g(h(x)) =L(M(x)+b(V+x))+a(U+M(x)+b(V+x))\\
        &=&L(M(x))+L(b(V+x))+a(U+M(x)+b(V+x)).
    \end{eqnarray*}
    Now $U+M(x)$ depends only on $M^{-1}(U)+x$. Indeed,
    since $M(M^{ -1}(U))\subset U$, $M$ induces a map $\overline{M}:\F_q/M^{ -1}(U)\rightarrow\F_q/U $ satisfying $\overline{M}(M^{ -1}(U)+x)=U+M(x)$.
    Therefore, we may write 
    $$f(x)=N(x)+c(x+W)$$ where $W=M^{ -1}(U)\cap V$
    for some map $c:\F_q/W\rightarrow\F_q$ and some $\F_p$-linearised polynomial $N=L\circ M$.
    Now $$\codim W\leq\codim M^{ -1}(U)+\codim V.$$
    Further $$\codim M^{-1}(U)\leq \codim U;$$ this is because 
    $$\dim M^{-1}(U)=\dim\ker M+\dim \im M\cap U,$$ which can be seen from the fact that $M$ induces an isomorphism 
    $$M^{-1}(U)/\ker M\rightarrow \im M\cap U,$$ and
    $$\dim \im M\cap U\geq \dim U-\codim \im M=\dim U-\dim \ker M.$$ We conclude
    $$\codim f\leq \codim M^{ -1}(U)\cap V\leq \codim U+\codim V$$ as desired.
\end{proof}
The case of the lemma that will occur in the paper is the one where $g$ or $h$ is a bijective
$\F_p$-affine map (in practice, they will even be $\F_q$-affine map).
\begin{lemma}
\label{lm:compo}
    Let $f,g,h$ be self-mappings of $\F_q$ such that $f=g\circ h$.
    If $g$ is a bijective $\F_p$-affine map,
    then $$\codim f=\codim h.$$  If $h$ is a bijective $\F_p$-affine map,
    then $$\codim f=\codim g.$$
\end{lemma}
\begin{proof}
    The first statement follows from Lemma \ref{lm:compoGen}, applied to 
    $$f=g\circ h\quad \mbox{and}\quad h=g^{-1}\circ f$$ together with the fact that $\codim g=\codim g^{-1}=0$.
    The second follows similarly, interchanging the roles of $g$ and $h$.
\end{proof}

\section{Main results: additive index and Carlitz rank cannot be small simultaneously}\label{main}

Let us show that a permutation of $\F_q$
of small positive Carlitz rank necessarily has large additive index.
First let us focus on the inversion itself. In characteristics $p>3$, it is particularly
easy to prove that the inversion has maximal index.
\begin{proposition}
\label{th:scalar}
    The inversion $f:x\mapsto x^{q-2}$ has additive index $q$ unless $p=2,3$.
\end{proposition}
\begin{proof}
Suppose there exists an $\F_p$-linear map $L$ and a subspace $U$ of codimension $k$
and a constant $a_V\in\F_q$ for each $V\in \cV=\F_q/U$
such that
$f(x)=L(x)+a_{U+x}$.
Consider $g\in\F_p\setminus \{0,1,-1\}$, which requires $p>3$. Then for every $V\in\cV$ and $x\in V\setminus\{0\}$, we have
$$g^{-1}x^{-1}=f(gx)=L(gx)+a_{gV}=gx^{-1}-ga_V+a_{gV},$$
so that
$$x^{-1}=(g^{-1}-g)^{-1}(a_{gV}-ga_V)$$ is constant on $V\setminus\{0\}$.
This implies that $V\setminus\{0\}$ is a singleton, that is, $|V|\le 2<p$. Thus $k=n$ and the additive index is $p^k=q$.
\end{proof}
For small characteristics the result is still essentially true but we are forced to use a different method and single out exceptions. The following theorem subsumes Proposition~\ref{th:scalar}, which we decided to keep because of the simplicity of the proof.
\begin{theorem}
\label{th:inversion}
    The inversion $f:x\mapsto x^{q-2}$ has additive index $q$, unless $q\leq 4$, in which case it has additive index $1$.
\end{theorem}
\begin{proof}
When $q\in\{2,3\}$, the inversion coincides with the identity map $x\mapsto x$.
When $q=4$, the inversion coincides with the $\F_2$-linear map $x\mapsto x^2$. In these cases, the additive index of the inversion is therefore 1.

Now there remains to prove that if the inversion has index less than $q$, then $q\leq 4$.

To do this, 
suppose there exists an $\F_p$-linear map $L$ and a subspace $U$ of codimension $k<n$
and a constant $a_V\in\F_q$ for each $V\in \cV=\F_q/U$
such that
$f(x)=L(x)+a_{U+x}$.
For any $V\in \cV$, denote $V^*=V\setminus\{0\}$.
Thus for any $(V,W)\in \cV^2$ and $(x,y)\in V^*\times W^*$ such that $x+y\neq 0$
we find that $$(x+y)^{-1}=L(x+y)+a_{V+W}=x^{-1}+y^{-1}+a_{V+W}-a_V-a_W,$$
so that
$(x+y)^{-1}-x^{-1}-y^{-1}$ is a constant $$c_{V,W}=a_{V+W}-a_V-a_W$$ when $(x,y)$ ranges in $V^*\times W^*$.
Multiplying both sides by $xy(x+y)$, we find
$$xy-y(x+y)-x(x+y)-c_{V,W}xy(x+y)=0,$$ 
or equivalently
$$x^2+y^2+c_{V,W}xy(x+y)+xy=0.$$
At fixed $y\neq 0$, the left-hand side is a non-zero polynomial $P_V$ of degree at most 2, which vanishes on all $V\setminus\{0,-y\}$, for every $V\in\cV$, thus $|V\setminus\{0,-y\}|\leq 2$.

If $k=0$, that is, $U=\F_q$, we infer $|\F_q\setminus\{0,-y\}|\leq 2$ thus 
$q\leq 4$ as desired.

If $0<k<n$, we can assume that $y\in U$ and $V\neq U$ so that $|V|\leq 2$.
Thus $p=2,k\geq n-1$, therefore $k=n-1$.

We must again conclude $q\leq 4$.
Thus $U=d\cdot \F_2$ for some $d\in\F_q^*$
and $\cV=\{b+U:b\in\cB\}$, where $\cB\subset\F_q$ is a set of representatives of $\F_q$ modulo $U$. We may choose that 0 is the representative of $U.$
Thus there exists an $\F_p$-linear map $L$ and elements $a_b$ for $b\in\cB$
such that $a_0=0$ (since $f(0)=0$) and for every $b\in \cB\setminus\{0\}$ 
$$(b+d)^{-1}=L(b+d)+a_b=L(d)+L(b)+a_b=L(d)+b^{-1}$$
for every $b\in\cB\setminus\{0\}$. Multiplying by $b(b+d)$
we find $L(d)b(b+d)=-d$. Because $d\neq 0$, this implies $L(d)\neq 0$.
So the polynomial $L(d)X(X+d)+d$ has degree $2$ and at least $|\cB\setminus\{0\}|=q/2-1$ roots (since $k=n-1$).
So $q/2-1\leq 2$ and since $q$ is a power of 2, we conclude $q\leq 4$ as desired.
\end{proof}

This result extends to all maps of Carlitz rank $1$, that is, maps of the form $$x\mapsto (ax+b)^{q-2}+c, \quad a\ne 0,b,c\in\F_q,$$ because of Lemma \ref{lm:compo}

We can now prove a refined version of \eqref{rankadd}.
\begin{theorem}
\label{th:lowrank}
    \thmtext
\end{theorem}
\begin{proof}

        
 First, we deal with the case $r=1$. By definition, $f$ is then of the form
 $x\mapsto (a_0x+a_1)^{q-2}+a_2$ with $a_0\neq 0$. Therefore, the result directly 
 follows from Theorem \ref{th:inversion} and Lemma \ref{lm:compo}.
 
 We now deal with the general case.
 By Lemma \ref{cr}, there are $\alpha,\beta,\gamma,\delta \in \F_q$ such that
    $$f(x)=\frac{\alpha x+\beta}{\gamma x+\delta}$$
    for at least $q-r$ different $x\in \F_q$.
    We may assume that 
    \begin{equation}\label{det}\alpha\delta\not=\beta\gamma
    \end{equation}
    since otherwise $\frac{\alpha x+\beta}{\gamma x+\delta}$ is constant, thus $r\ge q-1$
    and the result is trivial.
    
    Suppose there exists an $\F_p$-linear map $L$ and a subspace $U$ of codimension $k$
and a constant $a_V\in\F_q$ for each $V\in \cV=\F_q/U$
such that
$f(x)=L(x)+a_{U+x}$.
We distinguish two cases.
    
    If $\gamma=0$ and thus $\alpha\delta\not=0$ by \eqref{det}, upon applying Lemma \ref{lm:compo}, we may assume
    that $\alpha/\delta=1$ and $\beta=0$, that is, $f(x)=x$ for all but $r$
    elements.
Since $L$ is a linearised polynomial of degree at most $p^{n-k-1}$, the equality
$x=L(x)+a_V$  holds for at most $p^{n-k-1}$ elements of $V$ -- unless it holds for all $p^{n-k}$ elements of $V$. If it holds on all $x\in V$ for some $V$, we get $L(x)=x,a_V=0$. But then $f(x)=x+a_{U+x}$, which coincides with $x$ on either all or no elements of each coset. Since $r\neq 0$ by hypothesis, we know that $f$ is not globally an 
$\F_q$-affine map, hence we must have $r\geq p^{n-k}$, hence $p^k\geq q/r$.
Otherwise, for each $V$, the equality
$x=L(x)+a_V$  holds for at most $p^{n-k-1}$ elements of $V$. So $r\geq p^{n}-p^{n-1}\ge q/2$
and the result is trivial. 

If $\gamma\neq 0$, then
applying a linear transformation again, we may assume that $\gamma=1,\delta=0$,
thus $\beta\not=0$ by \eqref{det} and $f(x)=\alpha+\beta x^{-1}$ for all but $r$ elements.
Again using Lemma~\ref{lm:compo}, we may assume $\alpha=0,\beta=1$ so in fact $f(x)=x^{-1}$
for all but $r$ elements $x\in\F_q$.

Fix $y\in\F_q^*$.
Let $S=\{x\in \F_q: (f(x),f(x+y))=(x^{-1},(x+y)^{-1})\}$, a set of cardinality at least $q-2r$.
By the pigeonhole principle, there must be $V\in\cV$
    such that 
    \begin{equation}\label{VS}|V\cap S|\ge \frac{|S|}{p^k}\ge \frac{q-2r}{p^k}.
    \end{equation}
As in the proof of Theorem \ref{th:inversion}, we find when $x\in S\cap V$
that
$$(x+y)^{-1}-x^{-1}=c_{y,V}$$
where $$c_{y,V}=f(y)+a_{V+y}-a_V-a_{U+y}$$ does not depend on $x$.
Therefore $x=x+y+c_{y,V}x(x+y)$ whence $$c_{y,V}x^2+c_{y,V}yx+y=0$$ for $x\in V\cap S$, 
where the left-hand side is a polynomial of degree at most 2 in $x$ (the constant term $y$ is not zero) which vanishes on $V\cap S$.
Thus $|V\cap S|\le 2$, 
whence $p^k\ge (q-2r)/2$ by \eqref{VS}.
\end{proof}

\begin{remark}
    Actually $\min\{q/r,(q-2r)/2\}$ can be substituted by $q/(r+1)-1$.
    For $q\le 2(r+1)$ this is trivial and for $q>2(r+1)$ we have
    $$\frac{q-2r}{2}=\frac{q}{r+1}+\frac{(r-1)q}{2(r+1)}-r>\frac{q}{r+1}-1.$$
    This concludes the proof of equation \eqref{rankadd}.
\end{remark}

\section{The relation of the additive index to degree and weight}\label{degree}

Now we study the relationship between additive index and degree as well as between additive index and weight.

\subsection{Additive index and degree}
\begin{theorem}\label{crkai}
    Let $f$ be a self-mapping of $\F_q$ such that
    $\aind(f)>1$. Then 
    $$\deg(f)\aind(f)\geq q.$$ This bound is sharp: for every 
    integer $k\in (0,n]$, there exists an $f$ such that $\aind(f)= p^k=q/\deg (f)$.
\end{theorem}
Of course, an $\F_p$-affine map may have degree $0$ or $1$ or even $-\infty$,
so the restriction $\aind(f)>1$ is necessary.
\begin{proof}
    Let $k=\codim f$.
    By Propositions~\ref{th:coset} and \ref{alt} there exists an $\F_p$-subspace $U\leq\F_q$ of codimension $k$,
    a linearised polynomial $L$ of degree at most $p^{n-k-1}$
    and a self-mapping~$h$ of $\F_q$ satisfying $h(x+u)=h(x)$
    for all $(x,u)\in\F_q\times U$ such that
    $f(x)=L(x)+h(x)$.
    The equation $h(x)=h(0)$ has therefore at least $p^{n-k}$ solutions,
    yet the polynomial $h$ is not constant since $\aind(f)>1$, that is, $f$ is not $\F_p$-affine,
    so $\deg(h)\geq p^{n-k}$, hence also $\deg(f)\geq p^{n-k}$ as desired.

    The bound is sharp since for every $k$ and subspace $U\leq \F_q$ of codimension $k$, the function
    $$f(x)=\prod_{u\in U}(x-u)$$ has degree $|U|=p^{n-k}$ and codimension at most $\codim U=k$ since $f(x+u)=f(x)$ for all $(x,u)\in\F_q\times U$. The first part of the Theorem implies equality.
\end{proof}

\subsection{Additive index and weight}\label{secweight}

\begin{theorem}
\label{th:wtai}
Let $f$ be of additive index $p^k$. Then we have
$$\w(f)\le \left\{\begin{array}{cc} n+1, & k=0,\\\left(\frac{(n-k+1)^p-1}{n-k}\right)^k, & 1\le k<n. \end{array}\right.$$
\end{theorem}
{\em Proof.} For $k=0$, $f$ is, up to an additive constant, a linearised polynomial of degree at most $p^{n-1}$ and thus of weight at most $n+1$. 

For $k$ with $1\le k< n$, write 
$$f(x)=g(M(x))+L(x)$$
with linearised polynomials $M$ and $L$ and a polynomial $g$ with
$$\deg(M)=p^{n-k}, \quad \deg(L)\le p^{n-k-1}\quad \mbox{and}\quad\deg(g)\le p^k-1.$$
Write
$$g(x)=\sum_{i=0}^{p^k-1}a_ix^i,$$ thus
$$f(x)=a_0+a_1M(x)+L(x)+\sum_{i=2}^{p^k-1}a_i M(x)^i.$$
The term $a_1M(x)+L(x)$ is linearised of degree at most $p^{n-k}$ and thus of weight at most $n-k+1>1$.\\
For $i=0,\ldots,p^k-1$ write
$$i=i_0+i_1p+\ldots+i_{k-1}p^{k-1}\quad \mbox{with} \quad 0\le i_0,i_1,\ldots,i_{k-1}<p,$$
and let 
$$\sigma(i)=i_0+i_1+\ldots+i_{k-1}$$
be the {\em $p$-weight} of $i$. Since $M(x)^p$ has the same weight as $M(x)$ the weight of $M^i$ is at most $\w(M)^{\sigma(i)}\le (n-k+1)^{\sigma(i)}$.
Hence, we get 
\begin{eqnarray*}\w(f)&\le& \sum_{i=0}^{p^k-1}(n-k+1)^{\sigma(i)}=\sum_{i_0,i_1,\ldots,i_{k-1}=0}^{p-1}(n-k+1)^{i_0+i_1+\ldots+i_{k-1}}=\left(\sum_{j=0}^{p-1}(n-k+1)^j\right)^k\\
&=&\left(\frac{(n-k+1)^p-1}{n-k}\right)^k,
\end{eqnarray*}
which completes the proof.
\hfill $\Box$\\

\begin{remarks}
1. For $k=0$, we may take any linearised polynomial~$f$ of degree at most $p^{n-1}$, up to an additive constant, and all weights between $0$ and $n+1$ are attained.\\
2. For $1\le k< n$ the bound can be simplified to $\w(f)< n^{kp}$.\\
3. For $k>0$, fixed $p$ and  $k=o(n/\log n)$, more precisely,
$$k\le \frac{\ln p}{p}\frac{n}{\ln n},$$ 
the result is non-trivial, that is, $\w(f)<p^n$.\\
4. Conversely, for $f$ of codimension 
$$k\geq \frac{n}{\ln(n/3)},$$ 
Theorem \ref{th:wtai} says nothing about the weight of $f$. 
More precisely, for $k$ with 
$$n \frac{\ln p}{\ln(2^p-1)}\le k<n$$ 
we have
$$\left(\frac{(n-k+1)^p-1}{n-k}\right)^k\ge (2^p-1)^k\ge p^n$$
and for $k$ with
$$\frac{n}{\ln(n/3)} \le k<n \frac{\ln p}{\ln(2^p-1)}< \frac{2n}{3}$$
we have
$$\left(\frac{(n-k+1)^p-1}{n-k}\right)^k\ge (n-k)^{k(p-1)}> \left(\frac{n}{3}\right)^{k(p-1)}
\ge  p^n.$$

In fact this is natural as the following examples show that 
both very small and very large weight are possible for $f$ of a given large additive index.\\

Example 1.
The inversion $f_1(x)= x^{q-2}$ has additive index $q$ if $q>4$ by Theorem \ref{th:inversion} 
and weight $1$.\\ 



Example 2. The following family of functions $f_2$ satisfies
$$\w(f_2)>\aind(f_2).$$
For a prime $p$ and integers $n$ and $k$ with 
$$1\le k<n\quad\mbox{and}\quad (k,p)\not=(1,2)$$ let $U$ be a subspace of $\F_{p^n}$ of codimension $k$. Let $f_2$ be the self-mapping of~$\F_{p^n}$ defined by
$$f_2(x)=\left\{\begin{array}{cc}0,& x\not\in U,\\ 1, &x \in U,\end{array}\right.$$
of codimension at most $k$.
First we show that 
$$\aind(f_2)=p^k$$ 
with the exception $(k,p)=(1,2)$ where this codimension is $0$. More precisely,
assume that $f_2$ has codimension $0$, that is, 
$f_2(x)=L(x)+1$
where $L(x)$ is a (linearised) polynomial of degree at most $p^{n-1}$. Since $f_2(x)$ has $p^n-p^{n-k}$ zeros we must have $p^n-p^{n-1}\le p^n -p^{n-k}\le \deg(f_2)\le p^{n-1}$ and thus $(k,p)=(1,2)$.
It is easy to verify that $f_2$ is indeed of codimension~$0$ in the case~$(k,p)=(1,2)$.

Let us henceforth assume that $(k,p)\neq (1,2)$.
Again, since
$f_2(x)= 0$ for $p^n-p^{n-k}$ elements $x\in\F_{p^n}^*$, the degree of $f_2$ is at least $p^n-p^{n-k}$. Let $\ell$ denote the codimension of $f_2$ which satisfies $1\le \ell\le k$ under our assumption. Then we can write 
$$f_2(x)=g(M(x))+L(x)$$ with (linearised) polynomials $M$ of degree $p^{n-\ell}$, $L$ of degree at most $p^{n-\ell-1}$ and a polynomial $g$ of degree at most $p^\ell-1$. Hence, 
$$p^n-p^{n-k}\le \deg(f_2)\le \max\{p^n-p^{n-\ell},p^{n-\ell-1}\}=p^n-p^{n-\ell},$$ where we used $\ell\ge 1$ in the last step, thus $\ell\ge k$. 

Moreover, the weight of $f_2$ is at least
$$\frac{p^n-1}{p^n-1-(p^n-p^{n-k})}>p^k,\quad 1\le k<n,$$
by Lemma~\ref{wt}. (Note that $\deg(f_2)=p^n-p^{n-k}\le p^n-2$ for $k<n$ and the condition of Lemma~\ref{wt} is satisfied.)
Hence,  
$$\w(f_2)\ge p^k+1.$$  
Note that for $k=n$ we have $f_2(x)=1-x^{p^n-1}$ of weight only $2$.\\
    
\end{remarks}
\begin{remark}
In particular, Example 2 shows that the additive index and the weight may be very large simultaneously.
\end{remark}

In fact, as $q$ tends to infinity, almost all maps have additive index $q$ (see \cite[Corollary~9]{biwi26}) and weight at least $(1-\varepsilon)q$ for any $\varepsilon>0$, which can be verified as follows. 

The expected value of the weight is
$$q^{-q}\sum_{f\in \F_q[x],\deg(f)<q}\w(f)=q^{-q}\sum_{j=0}^q {q\choose j}(q-1)^jj=q-1,$$
where we used
$$(x-1)\sum_{j=0}^q{q\choose j}(x-1)^{j-1}j=(x-1)\left(\sum_{j=0}^q{q\choose j}(x-1)^j\right)'
=(x-1)(x^q)'=(x-1)qx^{q-1}$$
evaluated at $x=q$.
Now for any $\varepsilon>0$ let $N_\varepsilon$ be the number of $f$ with $\w(f)\le (1-\varepsilon)q$.
Then we have 
$$(q-1)q^q\le N_\varepsilon (1-\varepsilon)q+(q^q-N_\varepsilon)q$$
and thus
$$N_\varepsilon\le \frac{q^{q-1}}{\varepsilon}=o(q^q),\quad q\rightarrow \infty,$$
that is, almost all maps $f$ have weight $>(1-\varepsilon)q$.


\section{The discrete logarithm}\label{log}
For any polynomial of degree at least $2$, \eqref{deg} shows that Carlitz rank and degree cannot be both small and, if $f(x)\not=a+bx^{p^n-2}$, $b\not=0$, \eqref{weight} implies that Carlitz rank and weight cannot be both small. 
For any permutation polynomial which is not linearised,
Theorem~\ref{th:lowrank} shows that Carlitz rank and additive index cannot be both small
and Theorem~\ref{crkai} shows that degree and additive index cannot be both small.

Now we consider a permutation $f$ for which all four measures, degree, weight, additive index and Carlitz rank are large. The following permutation $f$ of $\F_{p^n}$ is closely related to the discrete logarithm of $\F_{p^n}$. Let $\zeta$ be a primitive element of $\F_{p^n}$ and $(1,\lambda,\ldots,\lambda^{n-1})$ be a polynomial basis of $\F_{p^n}$, that is, $\F_{p^n}=\F_p(\lambda)$.
For 
$$x=x_0+x_1p+\ldots+x_{n-1}p^{n-1}\quad\mbox{with}\quad 0\le x_0,x_1,\ldots,x_{n-1}<p$$
define $f$ by
\begin{equation}\label{dl}f(\zeta^x)=x_0+x_1\lambda+\ldots+x_{n-1}\lambda^{n-1},\quad x=0,1,\ldots,p^n-2,
\end{equation}
and 
\begin{equation}\label{dl0}f(0)=-1-\lambda-\ldots-\lambda^{n-1}.
\end{equation}

\begin{theorem}
The unique polynomial of degree at most $p^n-1$ defined by \eqref{dl} and \eqref{dl0} satisfies
$$\deg(f)\ge q-\frac{q}{p}-1,\quad \w(f)\ge q-\frac{q}{p},$$
$$\aind(f)\ge \frac{p^n}{n+2},$$
and
$${\rm Crk}(f)\ge p^n -3(2p)^{n/2},\quad p>2.$$
\end{theorem}
{\em Proof.}
For the lower bounds on degree and weight see \cite[Corollary 2 and Theorem 2]{wi02},
for the lower bound on the additive index see \cite[Theorem 2]{biwi26}. It remains to prove the lower bound on the Carlitz rank. 

Let $r$ denote the Carlitz rank of $f$ and assume $p>2$. As in the proof of Theorem~\ref{th:lowrank} we may assume that
\begin{equation}\label{moeb}f(x)=\frac{\alpha x+\beta}{\gamma x+\delta},\quad \alpha\delta\not=\beta\gamma
\end{equation}
for $M\ge p^n-r-1$ different $x\in \F_{p^n}^*$.

For $\gamma=0$, that is, $\alpha\delta\not=0$, from \cite[Theorem 3]{wi02} we get
$$M(M-2^n+1)\le 2^n(p^n-2^n)$$
and thus
$$p^n-r-1\le M< 2(2p)^{n/2}$$
and the result follows in this case.

For $\gamma\not=0$ we modify the proof of \cite[Theorem 3]{wi02} to get the result.  
We denote by ${\mathcal S}$ the set of $x\in \F_q^*$ satisfying \eqref{moeb} of size $M$.
Put 
  $${\mathcal W}=\{w_0+\ldots+w_{n-1}\lambda^{n-1} : w_i\in \{0,1\} \mbox{ for }
  0\leq i<n\}$$
  and consider
  $${\mathcal D}=\{\rho=\eta\xi^{-1} : -f(\rho)\in \F_q^*\setminus {\mathcal W},
	~\xi,\eta \in \mathcal{S}\}.$$
  Obviously, there exists $\rho\in {\mathcal D}$ such that there are at least
  $$\frac{M(M-|{\mathcal W}|)}{|{\mathcal D}|}\ge
	\frac{M(M-|{\mathcal W}|)}{q-1-|{\mathcal W}|}
	=\frac{M(M-2^n+1)}{q-2^n}
    $$
  representations $\rho=\eta\xi^{-1}$, $\xi,\eta \in \mathcal{S}$. 
  Choose this $\rho$
  and put 
  $${\mathcal R}=\left\{\xi \in \F_q^* : f(\xi)=\frac{\alpha\xi+\beta}{\gamma\xi+\delta} \mbox{ and } 
	f(\rho\xi)=\frac{\alpha\rho\xi+\beta}{\gamma\rho\xi+\delta}\right\}.$$
  We have
  $$|{\mathcal R}|\geq \frac{M(M-2^n+1)}{q-2^n}.$$
  By \cite[Lemma~1]{wi02} there are at most $2^n$ possible elements 
  $\omega \in \F_q$, namely the elements of~${\mathcal W}$,  
  such that
  $$\frac{\alpha\rho\xi+\beta}{\gamma\rho\xi+\delta}=f(\rho\xi)=f(\rho)+f(\xi)+\omega
    =f(\rho)+\frac{\alpha\xi+\beta}{\gamma\xi+\delta}+\omega.$$ 
  Therefore at least one of the $2^n$ rational functions
  $$h_\omega(X)=\frac{\alpha\rho X+\beta}{\gamma\rho X+\delta}-\frac{\alpha X+\beta}{\gamma X+\delta}-f(\rho)-\omega$$ has at least $|{\mathcal R}|/2^n$ 
  zeros in $\F_{p^n}$. By the choice of ${\mathcal D}$ 
  we have $-f(\rho)-\omega\neq 0$ and 
  all of these functions are not identical to zero. After clearing denominators we get a quadratic polynomial equation with at most two zeros. Thus 
  $$2\geq\frac{|{\mathcal R}|}{2^n}\geq 
	\frac{M(M-2^n+1)}{(p^n-2^n)2^n},$$
    thus
    $$p^n-r-1\le M<\sqrt{2}(2p)^{n/2}+2^n$$
    and the result follows.
    \hfill $\Box$
    
\section{Multiplicative index}\label{multind}
The multiplicative index is another complexity measure, arising from the theory of cyclotomic mappings, and measuring to which extent a mapping
is close to a monomial. 
Thus a monomial has multiplicative index 1, and a generic
mapping has multiplicative index $q$. 

More precisely, let $\ell$ be a positive divisor of $q-1$ and $\zeta$ a primitive element of $\F_q$. Then the set $C_0$ of
nonzero $\ell$th powers of elements of $\F_q$ is a subgroup of $\F_q^*$ of index $\ell$. The elements of the factor group $\F_q^*/C_0$ are
the cyclotomic cosets
$$C_i = \zeta^i C_0,\quad  i = 0, 1, \ldots, \ell-1.$$
For any positive integer $r$ and any $a_0, a_1, \ldots, a_ {\ell-1} \in \F_q^*$, we define the $r$-th order
cyclotomic mapping $f^r_{a_0,a_1,\ldots,a_{\ell-1}}$ of index $\ell$ by
\begin{equation}\label{minddef}
f^r_{a_0,a_1,\ldots,a_{\ell-1}} (x) =
\left\{\begin{array}{cc}
0 &\mbox{if }x = 0,\\
a_ix^r &\mbox{if }x \in C_i, 0\le   i <\ell.\end{array}\right.
\end{equation}
For a polynomial $f$ over $\F_q$ with $f (0) = 0$ we denote by ${\rm Ind}(f)$ the smallest index
$\ell$ such that $f(x)$ can be represented in the form \eqref{minddef}. If $f(0)\not= 0$, the multiplicative index of $f$ is the multiplicative index of the function $x\mapsto f(x)-f(0)$. The index was introduced by Wang in \cite{wa07}
modifying an earlier concept described in \cite{niwi05}. 

\subsection{Relations between multiplicative index and degree, weight or Carlitz rank}

The interpolation polynomial of a cyclotomic mapping of multiplicative index $\ell$ is of the form
\begin{equation}\label{intpol} (A_{\ell-1}x^{(\ell-1)(q-1)/\ell} +\ldots + A_1x^{(q-1)/\ell} + A_0)x^r\bmod (x^q-x),
\end{equation}
which can be obtained analogously to \cite[Theorem~1]{niwi05}, which is the case $r=1$. 
Note that we may assume $r\le (q-1)/\ell$ since otherwise
\begin{eqnarray*}&&(A_{\ell-1}x^{(\ell-1)(q-1)/\ell} +\ldots + A_1x^{(q-1)/\ell} + A_0)x^r\\
&\equiv& (A_{\ell-2}x^{(\ell-1)(q-1)/\ell} +\ldots + A_0x^{(q-1)/\ell} + A_{\ell-1})x^{r-(q-1)/\ell}\bmod (x^q-x).
\end{eqnarray*}
From \eqref{intpol} we immediately get 
$$\deg(f)\ge \frac{q-1}{{\rm Ind}(f)},\quad {\rm Ind}(f)>1,$$
that is, degree and index cannot be simultaneously small with the exception of functions $f(x)=ax^r$ with a small $r$.
We also get from \eqref{intpol} that
$$
\w(f)\le {\rm Ind}(f).$$

For any $a\in \F_q^*$ the mappings $x\mapsto ax$ and $x\mapsto ax^{q-2}$ have both multiplicative index~$1$ and Carlitz rank $0$ and $1$, respectively.
For any permutation $f$ of  $\F_q$ which does not coincide with such a map on more than $3q^{1/2}$ elements of $\F_q$ we have
$$\crk(f)\ge q-3\max\{{\rm Ind}(f),q^{1/2}\}$$
by \cite[Theorem 1]{iswi18}, that is, Carlitz rank and multiplicative index cannot be simultaneously small for any function which is not close to either a line or a hyperbola.

Finally, we mention that the function $f$ defined by \eqref{dl} and \eqref{dl0} has large multiplicative index 
$${\rm Ind}(f)\ge \frac{q-1}{6}.$$

{\em Proof.} The multiplicative index of this function is by definition the same as the multiplicative index of the function $f$ defined by
$f(\zeta^x)=(x_0+1)+(x_1+1)\lambda+\ldots+(x_{n-1}+1)\lambda^{n-1}$ for any $x\in\{0,\ldots,q-2\}$ together with $f(0)=0$, which is a permutation, that is,
we must have $\gcd(r,(q-1)/\ell)=1$.

Assume $\ell\le (q-1)/3$ and 
$$f(\zeta^{i\ell})=a\zeta^{i\ell r}\quad \mbox{for } i=0,1,\ldots,\frac{q-1}{\ell}-1\ge 2.$$
Taking $i=0,1$ and $(q-1)/\ell-1$ we get 
\begin{eqnarray*}
1+\lambda+\ldots+\lambda^{n-1}&=& a,\\
(\ell_0+1)+(\ell_1+1)\lambda+\ldots+(\ell_{n-1}+1)\lambda^{n-1}&=& a\zeta^{\ell r},\\
(p-\ell_0)+(p-\ell_1)\lambda+\ldots+(p-\ell_{n-1})\lambda^{n-1}&=&a\zeta^{-\ell r},
\end{eqnarray*}
where
$$\ell=\ell_0+\ell_1p+\ldots+\ell_{n-1}p^{n-1},\quad 0\le \ell_0,\ell_1,\ldots,\ell_{n-1}<p.$$
We get
$$\zeta^{\ell r}+\zeta^{-\ell r}-1=0$$
and thus $\zeta^{6\ell r}=1$, which implies, $6 r\equiv 0 \bmod (q-1)/\ell$. Since $\gcd(r,(q-1)/\ell)=1$ we must have $\ell\ge (q-1)/6$. \hfill $\Box$

\subsection{A conjecture on the relation between additive and multiplicative index}
In view of the heuristic that multiplicative objects cannot in general be additively special, and vice-versa, it is natural to imagine that
apart from a few exceptions, functions of low additive index must have high multiplicative index. Below, we formulate a conjecture, deliberately vaguely, since it seems difficult to guess, let alone prove, the right formulation.

\begin{conj}
    Apart from monomials of the form $x\mapsto ax^{p^j}$, which have additive and multiplicative indices equal to 1, any function with low multiplicative index has high additive index and vice versa.
\end{conj}

\begin{remark}

A small-scale version of this conjecture can be stated and proved as follows.\\ 
    
\end{remark}
    If $f(x)$ is not of the form $f(x)=ax^{p^j}$, then we have 
    $$\mbox{either }{\rm Ind}(f)> p \quad \mbox{or }\aind(f)\ge p.$$
    
{\em Proof.}
Since otherwise the result is trivial we may assume 
$\aind(f)=1$. Further assume 
$$f(x)=ax^r=L(x)+c,\quad x\in (C\cup\{0\}),$$
where $C$ is a multiplicative subgroup $C\subseteq\F_q^*$ of order $(q-1)/\ell$, $c\in \F_q$, $a\in\F_q^*$ 
and~$L(x)$ is a linearised polynomial of degree at most $p^{n-1}$. 
Since $f(0)=0$ we have 
$c=0$.

Put
$$F(x)=ax^r-L(x).$$
By assumption we may assume that $f(x)=L(x)$ is not of the form $L(x)=ax^{p^j}$. Hence, $n\ge 2$ and  $F$ is not identically zero. We have
$$\deg(F)=\max\{\deg(L),r\}\ge 1.$$
$F$ has at least $(p^{n}-1)/\ell+1$ zeros and thus
$$\max\{p^{n-1},r\}\ge  \frac{p^{n}-1}{\ell}+1.$$
By \eqref{intpol} we may assume $r\le (q-1)/\ell$, hence
$\ell\ge (p^n-1)/(p^{n-1}-1)>p$, $n\ge 2$. 
\hfill $\Box$

\end{document}